\documentclass{amsart}
\usepackage{graphicx}
\usepackage{amsmath,bm}
\usepackage{amssymb, comment}
\usepackage[all]{xy}
\usepackage{color}

\newcommand{\Hom}{\operatorname{Hom}\nolimits}

\newcommand{\Tor}{\operatorname{Tor}\nolimits}
\newcommand{\Ext}{\operatorname{Ext}\nolimits}

\newcommand{\n}{\operatorname{\mathfrak{n}}\nolimits}

\newcommand{\V}{\operatorname{V}\nolimits}
\newcommand{\W}{\operatorname{W}\nolimits}
\newcommand{\G}{\operatorname{G}\nolimits}

\newcommand{\T}{\operatorname{\mathcal{T}}\nolimits}

\newcommand{\comments}[1]{}

\newtheorem{theorem}{Theorem}[section]

\newtheorem{lemma}[theorem]{Lemma}
\newtheorem{proposition}[theorem]{Proposition}
\theoremstyle{definition}
\newtheorem{definition}[theorem]{Definition}

\theoremstyle{definition}

\theoremstyle{definition}

\theoremstyle{definition}
\newtheorem*{example}{Example}
\theoremstyle{definition}

\theoremstyle{definition}

\theoremstyle{definition}
\newtheorem*{remark}{Remark}

\theoremstyle{definition}

\begin{document}

\title[higher-order support varieties]{Realizability and the Avrunin-Scott theorem for higher-order support varieties}
\author{Petter Andreas Bergh \& David A.\ Jorgensen}
\address{Petter Andreas Bergh \\ Institutt for matematiske fag \\
  NTNU \\ N-7491 Trondheim \\ Norway}
\email{bergh@math.ntnu.no}
\address{David A.\ Jorgensen \\ Department of mathematics \\ University
of Texas at Arlington \\ Arlington \\ TX 76019 \\ USA}
\email{djorgens@uta.edu}


\begin{abstract}
We introduce higher-order support varieties for pairs of modules over a commutative local complete intersection ring, and give a complete description of which varieties occur as such support varieties.  In the context of a group algebra of a finite elementary abelian group, we also prove a higher-order Avrunin-Scott-type theorem, linking higher-order support varieties and higher-order rank varieties for pairs of modules.
\end{abstract}

\subjclass[2010]{13D02, 13D07, 16S34, 20C05}

\keywords{Complete intersections, support varieties}

\thanks{Part of this work was done while we were visiting the Mittag-Leffler Institute in February and March 2015. We would like to thank the organizers of the Representation Theory program.}

\maketitle

\section{Introduction}\label{sec:intro}

Support varieties for modules over commutative local complete intersections were introduced in \cite{Avramov} and \cite{AvramovBuchweitz}, inspired by the cohomological varieties of modules over group algebras of finite groups. These geometric invariants encode several homological properties of the modules. For example, the dimension of the variety of a module equals its complexity. In particular, a module has finite projective dimension if and only if its support variety is trivial. 

In this paper, we define higher-order support varieties for pairs of modules over complete intersections. These varieties are defined in terms of Grassmann varieties of subspaces of the canonical vector space associated to the defining regular sequence of the complete intersection. Thus, for a fixed dimension $d$, the support varieties of order $d$ are subsets of the Grassmann variety of $d$-dimensional subspaces of the canonical vector space, under a Pl\" ucker embedding into
$\mathbb P^{{c \choose d}-1}$. For $d=1$, we recover the classical support varieties: the varieties of order $1$ are precisely the projectivizations of the support varieties defined in \cite{AvramovBuchweitz}.

We show that several of the results that hold for classical support varieties also hold for the higher-order varieties. Among these is the realizability result: we give a complete description of the closed subsets of the Grassmann variety that occur as higher-order support varieties. We also prove a higher-order Avrunin-Scott result for group algebras of finite elementary abelian groups. Namely, we extend the notion of
$r$-rank varieties from \cite{CarlsonFriedlanderPevtsova} to higher-order rank varieties of pairs of modules and show that these varieties are isomorphic to the higher-order support varieties.

In Section 2 we give our definition of higher-order support varieties, and prove some of their elementary properties.  In particular, we show that they are well-defined, independent of the choice of corresponding intermediate complete intersection, and are in fact closed subsets of the Grassmann variety. In Section 3 we discuss the realizability question, and in Section 4 we prove the higher-order Avrunin-Scott result.

\section{Higher-order support varieties}\label{sec:hdsv}

In this section and the next, we fix a regular local ring $(Q, \n, k)$ and an ideal $I$ generated by a regular sequence of length $c$ contained in $\n^2$. We denote by $R$ the complete intersection ring 
$$R = Q/I,$$
and by $V$ the $k$-vector space 
$$V=I/\n I.$$
For an element $f\in I$, we let $\overline f$ denote its image in $V$.

If the codimension of the complete intersection $R=Q/I$ is at least 2, then $V$ has dimension at least 2, and it makes sense to consider subspaces $W$ of $V$.  Each such subspace has many corresponding complete intersections, in the following sense: if $W$ is a subspace of $V$, then choosing preimages in $I$ of a basis of $W$ we obtain another regular sequence \cite[Theorem 2.1.2(c,d)]{BrunsHerzog}, and the ideal 
$J\subseteq I$ it generates.  We thus get natural projections of complete intersections $Q\to Q/J\to R$. We call $Q/J$ a \emph{complete intersection intermediate to $Q$ and $R$}, or when the context is clear, simply an 
\emph{intermediate complete intersection}. 

We now give our definition of higher-order support variety. We fix a basis of $V$, and let $\G_d(V)$ denote the Grassmann variety of $d$th order subspaces of $V$ under the Pl\" ucker embedding into
$\mathbb P^{{c \choose d}-1}$ with respect to the chosen basis of $V$.

\begin{definition} We set
\[
V_R^d(M,N)=\{p_W\in\G_d(V)\mid \Ext_{Q/J}^i(M,N)\ne 0 \text{ for infinitely many $i$}\},
\]
where $W$ is a $d$th order subspace of $V$, $p_W$ is the corresponding point in the Grassmann variety $\G_d(V)$, and $Q/J$ is an intermediate complete intersection corresponding to $W$. 
We also define $\V_R^d(M)=\V_R^d(M,k)$.
\end{definition}

\begin{remark}
We note that $\V_R^1(M,N)$ is the projectivization of the affine support variety $\V_R(M,N)$
defined in \cite{AvramovBuchweitz}. 
\end{remark}

There are two aspects of the definition which warrant further discussion.
\begin{enumerate}
\item \label{independent}
The definition is independent of the chosen intermediate complete intersection $Q/J$ 
corresponding to $W$, and
\item \label{closed} $\V_R^d(M,N)$ is a closed set in $G_d(V)$.
\end{enumerate}
We next give proofs of these two statements.  

Let $Q/J$ and $Q/J'$ be two complete intersections intermediate to $Q$ and $R$.  The condition that  
\[
(J+\n I)/\n I=(J'+\n I)/\n I
\] 
in $V$ defines an equivalence relation on the set of such intermediate complete intersections.  The following result addresses (\ref{independent}) above.

\begin{proposition} Suppose that $Q/J$ and $Q/J'$ are equivalent complete intersections intermediate to $Q$ and $R$, that is, $(J+\n I)/\n I=(J'+\n I)/\n I$ in $V$.  Then for all finitely generated 
$R$-modules $M$ and $N$ one has
$\Ext_{Q/J}^i(M,N)=0$ for all $i\gg 0$ if and only if $\Ext_{Q/J'}^i(M,N)=0$ for all $i\gg 0$.
\end{proposition}

\begin{proof} Let $W=(J+\n I)/\n I$ and consider the natural map of $k$-vector spaces 
$\varphi_{J}:J/\n J\to W\subseteq V$ 
defined by $f+\n J \mapsto f+\n I$. This is an isomorphism: it is onto by construction,  and one-to-one since $J\cap \n I=\n J$. The condition that $(J+\n I)/\n I=(J'+\n I)/\n I$ is equivalent to
$\varphi_J(J/\n J)=\varphi_{J'}(J'/\n J')$.  By \cite[Proposition 3.2]{BerghJorgensen}, one has the equality $\varphi_J(\V_{Q/J}(M,N))=\varphi_{J'}(\V_{Q/J'}(M,N))$, where $\V_{Q/J}(M,N)$ denotes the affine support variety of $M$ and $N$ over the complete intersection $Q/J$. By \cite[Proposition 2.4(1) and Theorem 2.5]{AvramovBuchweitz} one has that $\Ext_{Q/J}^i(M,N)=0$ for all $i\gg 0$ if and only if 
$\V_{Q/J}(M,N)=\{0\}$.  The same holds over $Q/J'$, and thus the result follows by the injectivity of 
$\varphi_{J}$.
\end{proof}

Next, we address the second point in the remark.

\begin{proposition}\label{proposition:closed}
For all finitely generated $R$-modules $M$ and $N$ one has that $\V_R^d(M,N)$ is a closed set in $G_d(V)$.
\end{proposition}

\begin{proof} This result follows from an incidence correspondence (see, for example, 
\cite[Example 6.14]{Harris}), as we now describe.
Set 
\[
\Gamma =\{(p_W,x)\in\G_d(V)\times \G_1(V) \mid x\in W\cap\V_R^1(M,N)\}.
\] 
Since $\Gamma$ is an incidence correspondence, it is a closed subset of the product space $\G_d(V)\times \G_1(V)$. We have the two natural projections
\[
\xymatrixrowsep{2pc}
\xymatrixcolsep{0pc}
\xymatrix
{
& \G_d(V)\times \G_1(V) \ar[dl]_\pi \ar[dr]^{\pi'}& \\
\G_d(V) & & \G_1(V)
}
\]
Now by classical results from elimination theory (see, for example, \cite[Theorem 14.1]{Eisenbud}), the image of $\Gamma$ under $\pi$ is closed in $\G_d(V)$. It suffices now to show that $\pi(\Gamma)=\V_R^d(M,N)$.  

We have $p_W\in\pi(\Gamma)$ if and only if
$x\in\V_R^1(M,N)$ for some $x\in W$. This is equivalent to $\Ext_{Q/(f)}^i(M,N)\ne 0$ for infinitely many $i$ and for some $f\in I$ 
with $\overline f=x\in W$. By \cite[Proposition 2.4(1) and Theorem 2.5]{AvramovBuchweitz}, this condition is the same as $\Ext_{Q/J}^i(M,N)\ne 0$ 
for infinitely many $i$.  By definition, this happens if and only if $p_W\in\V_R^d(M,N)$.
\end{proof}

\begin{remark}
(1) Let $\T=\{(p_W,x)\in \G_d(V)\times \G_1(V) \mid x\in W\}$. Then the map $\tau:\T \to G_d(V)$ given by
$\tau(p_W,x)=p_W$ is the tautological bundle over the Grassmann variety $\G_d(V)$.  For $\Gamma$ 
as in the proof of Proposition \ref{proposition:closed}, we have $\Gamma\subseteq\T$, and 
$\tau(\Gamma)=\V_R^d(M,N)$.  Thus $\V_R^d(M,N)$ may be interpreted as the image under the tautological bundle of the fiber of $\V_R^1(M,N)$ in $\T$.

(2) In the definition of $\V_R^d(M,N)$, a specific basis of $V$ was chosen.  We remark that the definition 
is independent of the choice of basis, in the sense that if another basis of $V$ is chosen, then the two higher-order support varieties are isomorphic.  Indeed, this is true for the first order affine varieties
$\V_R(M,N)$ by \cite[Remark 2.3]{AvramovBuchweitz}.  It then follows that the same is true for the 
projectivizations $\V_R^1(M,N)$, namely, there is an automorphism $\xi:\G_1(V) \to \G_1(V)$ such that 
if $\V_R^1(M,N)$ is the support variety with respect to the first basis, and $\V_R^1(M,N)'$ is the support variety with respect to the second, then $\xi(\V_R^1(M,N))=\V_R^1(M,N)'$. The general result for the higher-order support varieties follows from the incidence correspondence from the proof above.
\end{remark}

We now give basic properties of higher-order support varieties, akin to those of the one-dimensional affine support varieties.

\begin{theorem} \label{thm:props}
The following hold for finitely generated $R$-modules $M$ and $N$.
\begin{enumerate}
\item $\V_R^d(k)=\G_d(V)$. 
\item $\V_R^d(M,N)=\V_R^d(N,M)$.  For $d=1$, we moreover have $\V_R^1(M,N)=\V_R^1(M)\cap\V_R^1(N)$.
\item $\V_R^d(M,M)=\V_R^d(k,M)=\V_R^d(M)$.
\item If $M'$ is a syzygy of $M$ and $N'$ is a syzygy of $N$, then $\V_R^d(M,N)=\V_R^d(M',N')$. \label{syzygy}
\item If $0\to M_1\to M_2\to M_3\to 0$ and $0\to N_1\to N_2\to N_3\to 0$ are short exact sequences of finitely generated $R$-modules, then for $\{h,i,j\}=\{1,2,3\}$ there are inclusions
\[
\V_R^d(M_h,N)\subseteq\V_R^d(M_i,N)\cup\V_R^d(M_j,N);
\]
\[
\V_R^d(M,N_h)\subseteq\V_R^d(M,N_i)\cup\V_R^d(M,N_j).
\]
\item If $M$ is Cohen-Macaulay of codimension  $m$, then 
\[
\V_R^d(M)=\V_R^d(\Ext_R^m(M,R)).
\]
In particular, if $M$ is a maximal Cohen-Macaulay $R$-module, then
$\V_R^d(M)=\V_R^d(\Hom_R(M,R))$.
\item \label{regseq} If $x_1,\dots,x_d$ is an $M$-regular sequence, then 
\[
\V_R^d(M)=\V_R^d(M/(x_1,\dots,x_d)M).
\]
\end{enumerate}
\end{theorem}

\begin{proof}  The proof of properties (1)--(7) for the affine one-dimensional support varieties
$\V_R(M,N)$ are given in \cite{AvramovBuchweitz} (see also \cite{BerghJorgensen}.)  Since $\V_R^1(M,N)$ 
is simply the projectivization of $\V_R(M,N)$, the same properties also hold for these varieties. Finally,
properties (1)--(7) for $d>1$ follow from the $d=1$ case, as we now indicate.  

For a subset $X$ of $\G_1(V)$, we let
\[
\Gamma(X)=\{(p_W,x)\in\G_d(V)\times\G_1(V)\mid x\in W\cap X\}.
\]  
The proofs make repeated use of the fact 
that $\V_R^d(M,N)=\pi(\Gamma(\V_R^1(M,N)))$, where $\pi$ is as in the proof of Proposition \ref{proposition:closed}.  For example, for (1) we have 
$\V_R^d(k)=\pi(\Gamma(\V_R^1(k)))=\pi(\G_1(V))=\G_d(V)$.

For (2), we use the fact that $\V_R^1(M,N)=\V_R^1(M)\cap\V_R^1(N)=\V_R^1(N,M)$.  Therefore
$\V_R^d(M,N)=\pi\left(\Gamma\left(\V_R^1(M,N)\right)\right)=\pi\left(\Gamma\left(\V_R^1(N,M)\right)\right)
=\V_R^d(N,M)$

\sloppy To prove (3), we use the equalities $\V_R^d(M,M)=\pi(\Gamma(\V_R^1(M,M)))=\pi(\Gamma(\V_R^1(k,M)))=\V_R^d(k,M)$. The 
remaining equality and (4) are proved similarly.

To prove (5), we use the fact that for subsets $X$ and $Y$ of $\G_1(V)$ one has
$\Gamma(X\cup Y)=\Gamma(X)\cup\Gamma(Y)$. (We also use the fact that $\pi$ preserves unions, and both $\pi$ and $\Gamma$ preserve containment.) Therefore 
\begin{align*}
\V_R^d(M_h,N)&=\pi(\Gamma(\V_R^1(M_h,N)))\\
&\subseteq\pi(\Gamma(\V_R^1(M_i,N)\cup\V_R^1(M_j,N)))\\
&=\pi(\Gamma(\V_R^1(M_i,N)))\cup\pi(\Gamma(V_R^1(M_j,N)))\\
&=\V_R^d(M_i,N)\cup\V_R^d(M_j,N).
\end{align*}

The proofs of (6) and (7) are analogous to the proofs of \cite[Theorem 5.6(10)]{AvramovBuchweitz} and 
\cite[7.4]{AvramovIyengar} (see also \cite[Theorem 2.2(7) and (8)]{BerghJorgensen}.)
\end{proof}

We can extend Proposition 2.4(1) of \cite{AvramovBuchweitz}, to a sort of generalized Dade's Lemma, in the projective context.

\begin{proposition}
Fix $1\le d\le c$.  Then $\Ext_R^i(M,N)=0$ for all $i\gg 0$ if and only if $\V_R^d(M,N)=\emptyset$.
\end{proposition}

\begin{proof} 
By \cite[Proposition 2.4(1) and Theorem 2.5]{AvramovBuchweitz}, $\Ext_R^i(M,N)=0$ for all $i\gg 0$ if and only if $\V_R^1(M,N)=\emptyset$. The latter holds if and only if $\Gamma=\Gamma(\V_R^1(M,N))=\emptyset$, which in turn holds if and only if $\V_R^d(M,N)=\pi(\Gamma)=\emptyset$, where $\Gamma$ and $\pi$ are from the proof of Proposition \ref{proposition:closed}.
\end{proof}

\section{Realizability}

In this section we give a complete description of which closed subsets of $\G_d(V)$ can possibly occur 
as the $d$th order support variety $\V_R^d(M,N)$ of a pair of finitely generated $R$-modules $(M,N)$.  The basis of the description is the following result in the first order case.

\begin{theorem}\label{dimensionone}
Every closed subset of $\G_1(V)$ is the support variety of some finitely generated $R$-module.  Specifically, if $Z$ is a closed subset of $\G_1(V)$, then there exists a finitely generated $R$-module $M$ such that
$Z=\V_R^1(M,k)$.
\end{theorem} 

\begin{proof} This is well-known in the affine case, see, for example, \cite{Bergh}.  Since every closed set in 
$\G_1(V)$ is the projectivization of a cone in $V$, and $\V_R^1(M,N)$ is the projectivization of $\V_R(M,N)$, the result follows.
\end{proof}

The framework of the proof of Proposition \ref{proposition:closed} allows us to complete the description of realizable higher-order varieties.
Recall that $\pi$ denotes the projection map $\G_d(V)\times\G_1(V)\to\G_d(V)$.

\begin{theorem}\label{dimensiond} For a closed subset $Z$ of $\G_1(V)$, set 
\[
\Gamma(Z)=\{(p_W,x)\in\G_d(V)\times\G_1(V) \mid x\in W\cap Z\}.
\]
Let $Y$ be a closed subset of $\G_d(V)$.  Then $Y=\V_R^d(M,N)$ for a pair of finitely generated $R$-modules 
$(M,N)$ if and only if $Y=\pi(\Gamma(Z))$ for some closed subset $Z$ of $\G_1(V)$.
\end{theorem}

\begin{proof} Suppose that $Y=\V_R^d(M,N)$ for a pair of finitely generated $R$-modules $(M,N)$.  Then the proof of Proposition \ref{proposition:closed} shows that $Y=\pi(\Gamma(\V_R^1(M,N)))$.

Conversely, suppose that $Y=\pi(\Gamma(Z))$ for some closed subset $Z$ of $\G_1(V)$.  Then Theorem
\ref{dimensionone} shows that $Z=\V_R^1(M,N)$ for some pair of finitely generated $R$-modules $(M,N)$. 
Thus $Y=\pi(\Gamma(\V_R^1(M,N)))=\V_R^d(M,N)$, again from the proof of Proposition \ref{proposition:closed}.
\end{proof}

Theorem \ref{dimensiond} shows that, in contrast to first order support varieties, the realizability of varieties in $\G_d(V)$ for $d>1$ as $d$th order support varieties of a pair of finitely generated $R$-modules is more restrictive.  Indeed, consider a smallest nontrivial first order support variety $\V_R^1(M,N)$, namely, one consisting of a single point $x$.  Then $\V_R^d(M,N)$ consists of all $d$-dimensional planes in $V$ 
containing $x$.  Changing the basis of $V$ if necessary, we can assume that 
$x=(1,0,\dots,0)\in\G_1(V)$.  Then there is an obvious bijective correspondence between $d$-dimensional subspaces of $V$ containing $x$, and $(d-1)$-dimensional subspaces of a $(c-1)$-dimensional $k$-vector space. Thus $\dim\V_R^d(M,N)=\dim \G_{d-1}(k^{c-1})=(d-1)(c-d)$.  In particular, we have $\dim\V_R^{c-1}(M,N)=c-2$,
which is of codimension one in $\G_{c-1}(V)$, and this is when $\V_R^1(M,N)$ is nontrivially as small as possible.

The following example illustrates the previous discussion.

\begin{example} Let $k$ be a field (of arbitrary characteristic), and $Q=k[[x_1,\dots,x_c]]$.  Then $Q$ is a regular local ring with maximal ideal $\n=(x_1,\dots,x_c)$.  For $I=(x_1^2,\dots,x_c^2)$, the quotient ring
$R=Q/I$ is a codimension $c$ complete intersection.  Let $M=R/(x_1)$.  
Then it is not hard to show that relative to the basis 
$\overline {x_1^2},\dots,\overline {x_c^2}$ of $V=I/\n I$, the one-dimensional support variety of $M$ is 
$\V^1_R(M,k)=\{(1,0,\dots,0)\}$.  Thus we have $\dim\V_R^d(M,k)=(d-1)(c-d)$, for $1\le d\le c-1$.
\end{example}

\section{Higher-order rank varieties and a higher-order Avrunin-Scott Theorem} 

In this final section we consider complete intersections of a special form, namely, those which arise as the group algebra $kE$ of a finite elementary abelian $p$-group $E$, where $k$ has characteristic $p$. In this case one has 
\[
kE\cong k[x_1,\dots,x_c]/(x_1^p,\dots,x_c^p).
\] 
Note that by assigning $\deg x_i=1$ for $1\le i \le c$, the 
$k$-algebra $kE$ is standard-graded.  We let $kE_1$ denote the degree-one component of $kE$; this is a 
$k$-vector space of dimension $c$.  For any linear form 
$u$ of $kE_1$ one has $u^p=0$, and thus the subalgebra $k[u]$ of $kE$ generated by $u$ is isomorphic to 
$k[x]/(x^p)$ (for $x$ an indeterminate).  Since $k[u]$ is a principal ideal ring, every finitely generated 
$k[u]$-module is a direct sum of a free module and a torsion module. Recall from \cite{Carlson} that the \emph{rank variety} 
$\W_E(M)$ of a $kE$-module $M$ is the set of those linear forms $u\in kE_1$ such that the torsion part of $M$ as a $k[u]$-module is nonzero.  It was conjectured by Carlson \cite{Carlson} and subsequently proven by Avrunin and Scott \cite{AvruninScott} that the rank variety and the group cohomological support variety $\V_{kE}(M)$ of a $kE$-module agree. 

Recall that $I$ denotes the ideal $(x_1^p,\dots,x_c^p)$, and $V$ the $k$-vector space 
$I/\n I$, where $\n$ is the maximal ideal $(x_1,\dots,x_c)$.
We now want to show that the classical Avrunin-Scott theorem mentioned above is a special case of a more general result involving the higher-order varieties.  We generalize the definition of $d$th order rank varieties from 
\cite{CarlsonFriedlanderPevtsova} (which they call $d$-rank varieties) to $d$th order rank varieties
$\W^d_E(M,N)$ of pairs of modules $(M,N)$.  Fix a basis of $kE_1$, and consider the Grassmann variety 
$\G_d(kE_1)$ of $d$-dimensional subspaces of $kE_1$ under the Pl\" ucker embedding into 
$\mathbb P^{{c \choose d}-1}$ with respect to the chosen basis.

\begin{definition} We set 
\[
\W^d_E(M,N)=\{p_W\in \G_d(kE_1) \mid \Ext^i_{k[W]}(M,N)\ne 0 \text{ for infinitely many $i$}\}
\]
where $\G_d(kE_1)$ is the Grassmann variety of $d$-dimensional subspaces of $kE_1$, $p_W$ 
is the point in $\G_d(kE_1)$ corresponding to the $d$-dimensional subspace $W$, and $k[W]$ is the subalgebra of $kE$ generated by $W$.
\end{definition}

Consider the Frobenius isomorphism $\Phi:k \to k$ given by $\Phi(a)=a^p$. We have a $\Phi$-equivariant isomorphism of $k$-vector spaces 
\[
\alpha: kE_1 \to V
\] 
defined as follows.  For $u\in kE_1$, we 
choose a preimage $\widetilde u$ in $Q$, and then we set $\alpha(u)=\widetilde u^p+\n I\in V$.  
It is clear that $\alpha$ is a $\Phi$-equivariant homomorphism of $k$-vector spaces, which is defined
independent of the choice of preimage.  Since $k$ is algebraically closed, it contains $p$th roots, and
therefore $\alpha$ is onto.  Since $\dim kE_1=\dim V$, $\alpha$ is also one-to-one.

Taking as a basis for $V$ the image under $\alpha$ of the chosen basis of $kE_1$, we obtain an induced 
$\Phi$-equivariant isomorphism of Grassmann varieties
\[
\beta:\G_d(kE_1) \to \G_d(V)
\]
with respect to these bases.  Specifically, let $p_W$ be a point in 
$\G_d(kE_1)$, and $W$ the associated $d$-dimensional subspace of $kE_1$. Let $\widetilde W^p$ 
denote the ideal of $Q$ generated by the $p$th powers of linear preimages in $Q$ of a basis
of $W$.  Then $\beta(p_W)$ is the point in $\G_d(V)$ (with respect to the chosen basis of $V$) corresponding to the subspace $\widetilde W^p+\n I/\n I$.

\begin{theorem} \label{ASthm} Given finitely generated $kE$-modules $M$ and $N$, one has 
\[
\beta(\W_E^d(M,N))=\V^d_{kE}(M,N).
\]
\end{theorem}

The proof relies on the following lemma, which is a statement extracted from the proof of \cite[Theorem (7.5)]{Avramov}. For completeness we include the proof here. 

For any non-zero linear form $u\in kE_1$ we choose a preimage 
$\widetilde u$ in $Q=k[x_1,\dots,x_c]$, which is also a linear form, and define a 
homomorphism from $\mu :k[u]\to Q/(\widetilde u^p)$ by sending $u$ to $\widetilde u+(\widetilde u^p)$.  
Note that $Q/(\widetilde u^p)$ is free when regarded as module over $k[u]$ via $\mu$.  We have the commutative diagram of ring homomorphisms 
\[
\xymatrixrowsep{2pc}
\xymatrixcolsep{1.9pc}
\xymatrix{
 & Q/(\widetilde u^p) \ar[d]\\
k[u] \ar@{^{(}->}[r] \ar[ur]^\mu & kE 
}
\]
where the vertical map is the natural projection.  In particular, the action of $k[u]$ on a $kE$-module $M$ factors through $\mu$.

\begin{lemma} Let $M$ be a finitely generated $kE$-module.  Then $M$ has finite projective dimension over 
$k[u]$ if and only if it has finite projective dimension over $Q/(\widetilde u^p)$.
\end{lemma}

\begin{proof} The proof follows part of that of \cite[Theorem (7.5)]{Avramov}. 
Suppose that $M$ has finite projective dimension over $Q/(\widetilde u^p)$.  
Since $Q/(\widetilde u^p)$ is free over $k[u]$ any free resolution of $M$ over 
$Q/(\widetilde u^p)$ is also one of  $M$ over $k[u]$. Thus $M$ has a finite free resolution over
$k[u]$.  

Conversely, suppose $M$ is free as a $k[u]$-module.  Let $F$ be a minimal free resolution of $M$ over 
$Q/(\widetilde u^p)$.  Since $F$ is also a free resolution of $M$ over 
$k[u]$ and $\Tor_i^{k[u]}(M,k)=0$ for all $i>0$, we see that $F\otimes_{k[u]}k$ is a minimal free resolution of $M \otimes_{k[u]}k$ over $Q/(\widetilde u^p)\otimes_{k[u]}k\cong Q/(\widetilde u)$.  Since 
$Q/(\widetilde u)$ is regular and $F\otimes_{k[u]}k$ is a minimal, we must have that $F_c\otimes_{k[u]}k=0$, and this implies $F_c=0$.  Thus $F$ is a finite free resolution, and so $M$ has finite projective dimension 
over $Q/(\widetilde u)$.
\end{proof}

We now give a proof of Theorem \ref{ASthm}.

\begin{proof} 
Suppose that $p_W\in\W_E^d(M,N)$.  Then by definition there exist infinitely many nonzero 
$\Ext^i_{k[W]}(M,N)$.  Therefore, by Dade's Lemma, there exist infinitely nonzero 
$\Ext_{k[u]}^i(M,N)$ for some linear form $u\in W$.  Thus both $M$ and $N$ have infinite projective dimension over $k[u]$.  Therefore, by the lemma, both $M$ and $N$ have infinite projective dimension over 
$Q/(\widetilde u^p)$, and so it follows from \cite[Proposition 5.12]{AvramovBuchweitz} that there exist infinitely many nonzero $\Ext^i_{Q/(\widetilde u^p)}(M,N)$. This implies that there exist infinitely many nonzero $\Ext^i_{Q/(\widetilde W^p)}(M,N)$, where $\widetilde W^p$ represents the ideal generated by the $p$th powers of linear preimages in $Q$ of a basis of $W$. This gives
$\beta(p_W)\in\V^d_{kE}(M,N)$.  

For the reverse containment we just retrace our steps, noting that any $f\in I$ is equivalent mod $\n I$ to 
an element of the form $a_1x_1^p+\cdots+a_cx_c^p=(\sqrt[p]{a_1}x_1+\cdots+\sqrt[p]{a_c}x_c)^p$, $a_i\in k$,
and hence it is clear how to employ the previous lemma.
\end{proof}

\end{document}